\documentclass[12pt,a4paper]{article}

\usepackage{amssymb}
\usepackage{amsmath}
\usepackage{indentfirst}
\usepackage{t1enc}
\usepackage{array}
\usepackage{graphpap}
\usepackage[mathscr]{eucal}
\usepackage{enumerate}
\usepackage{tabularx}
\usepackage{graphicx}
\usepackage[colorlinks=true,linkcolor=blue,citecolor=blue,
  urlcolor=blue]{hyperref}
\usepackage[hyphenbreaks]{breakurl}
\usepackage{amsthm}
\usepackage{ulem}

\usepackage{tikz} 
\usepackage{pgfplots} 
\usetikzlibrary{shapes,arrows}

\usepackage[latin2,utf8]{inputenc}
\usepackage{float}

\theoremstyle{plain}
\newtheorem{thm}{Theorem}

\newtheorem{lem}{Lemma}
\newtheorem{dfn}{Definition}

\newtheorem{cns}{Construction}

\newcommand{\mc}{\mathcal}

\newcommand{\ab}[1]{\left\vert{#1}\right\vert}
\newcommand{\zj}[1]{\left({#1}\right)}

\newcommand{\lb}[1]{\label{#1}}

\allowdisplaybreaks[1] 

\begin{document}

\title{On the pseudorandom properties of filtered Legendre symbol sequences using three polynomials}

\author{Katalin Gyarmati and Károly Müllner}
\date{}

\maketitle

\footnotetext{
\indent Research supported by the Hungarian National Research Development and 
Innovation Fund KKP133819.}

\begin{abstract}
This paper presents a further development of a well-known construction that relies on polynomials and the Legendre symbol. To address theoretical security concerns regarding the original method, which used a single polynomial, we introduce a new approach that combines three different polynomials to generate the sequence, thereby enhancing its security. We prove that the sequences produced by this new construction also exhibit strong pseudorandom properties with respect to the pseudorandom measures introduced by Mauduit and Sárközy.
\end{abstract}  

\section{Introduction}

In 1997, Mauduit and Sárközy introduced the following quantitative
measures to study the pseudorandomness of finite binary sequences.

\begin{dfn}\lb{dfn01}
For a binary sequence
\[
E_N=\{e_1,\dots,e_N\}\in\{-1,+1\}^N,
\]
define
the \textit{well-distribution measure of} $E_N$ as
\[
W(E_N)=\max_{a,b,t}\ab{\sum_{j=1}^{t}e_{a+jb}},
\]
where the maximum is taken over all $a,b,t$ such that $a\in\mathbb Z$,
$b,t\in\mathbb N$ and $1\le a+b\le a+tb\le N$, while the \textit{correlation
measure of order $\ell$ of $E_N$} is defined as
\[
C_{\ell}(E_N)=\max_{M,D}\ab{\sum_{n=1}^{M} e_{n+d_1}
\dots e_{n+d_{\ell}}},
\]
where the maximum is taken over all $D=(d_1,\dots,d_{\ell})$ and $M$ such that
$0\le d_1<\dots <d_{\ell}<M+d_{\ell}\le N$.
\end{dfn}

These measures characterize the essential
random properties of binary sequences in
various applications (such as cryptography, Monte
Carlo methods, and many others). 
According to papers by Cassaigne, Mauduit, Sárközy \cite{CMS_On} and later
Alon, Kohayakawa, and Mauduit, \mbox{Moreira}, and R\"odl \cite{AKMMR},
the pseudorandomness of a sequence $E_n$ is considered to be very strong if
\begin{align*}
  W(E_N) &\ll N^{1/2}(\log N)^c,\\
  C_{\ell} (E_N) &\ll N^{1/2} (\log N)^{c_{\ell}}
\end{align*}
hold at least for small $\ell$'s. 
In various applications, pseudorandom constructions are of great importance.
Before Mauduit and Sárközy introduced the well-distribution and
correlation measures, these constructions were tested a posteriori.
This meant that after generation, computers were used to check
if certain statistical properties were satisfied.
According to the paper by Sárközy and Rivat \cite{RS},
and later a
paper of Mérai, Rivat and Sárközy \cite{MeRS} these
a posteriori tests can be avoided if the well-distribution
and correlation
measures of the sequence are small. Thus, it has become important to
construct sequences for which the $W$ and $C_k$ measures are
provably small.
In the literature, there are numerous constructions with
strong pseudorandomness properties (e.g., see
\cite{Ch},  \cite{CLX}, \cite{ChenW0}, \cite{GyPS}, 
\cite{L1}, \cite{L3}, \cite{LRS}, \cite{MRS2},
\cite{M4}, \cite{M3}, \cite{M1}, \cite{RS2}).
The following construction is perhaps the most natural and still the most
widely used today among all constructions proposed.

\begin{cns}[Hoffstein, Liemann]\lb{cns01}
    Let $p$ be an odd prime  and $f(x)\in\mathbb F_p[x]$ be a polynomial
   of degree $k$. Define $E_p = \{e_1,\dots,e_p\}\in\{-1,+1\}^p$ by:
\begin{equation*}
e_n=\left\{
\begin{array}{ll}
\zj{\frac{f(n)}{p}} & \textup{for } (f(n),p)=1,\\
+1 & \textup{for } p\mid f(n).
\end{array}
\right.
\end{equation*}
\end{cns}  

This construction was introduced by Hoffstein and Liemann, but nothing
has been proven about the pseudorandom properties of the sequences.
One year later, Goubin, Mauduit and Sárközy \cite{GMS} 
proved the following:

\bigskip
\textbf{Theorem A }[Goubin, Mauduit, S\'ark\"ozy]
  \textit{Let $p$ be an odd prime  and $f(x)\in\mathbb F_p[x]$ be a polynomial
  of degree $k$, which is not of the form
   $cq(x)^2$, where $c\in\mathbb F_p$, $q(x)\in\mathbb F_p [x]$.
   Define $E_p = \{e_1,\dots,e_p\}\in\{-1,+1\}^p$ by Construction \ref{cns01}.
 Then, }
\[
W(E_p)\le 10 kp^{1/2}\log p.
\]
\textit{Assume that one of the following three conditions for $\ell$,
which is the order of the correlation, holds:\\ 
(i) $\ell=2$;\\
(ii) $\ell<p$ and 2 is a primitive root modulo p;\\
(iii) $(4k)^{\ell}<p$.\\
Then, }
\[
C_{\ell}(E_p)\le 10 k\ell p^{1/2}\log p.
\]

\bigskip
Although Construction \ref{cns01} has strong pseudorandom properties, a potential weakness exists. It is conceivable that a future algorithm could determine
the polynomial $f$ from just $p^{\varepsilon}$ consecutive elements of the
sequence $E_p$, especially if the polynomial's degree is under a certain bound.
In this case, the entire sequence
(e.g., used as the secret key) might be determined
from a few elements of the sequence.
A solution to this problem could be to modify the sequence by using not one, but three different polynomials.
For this, we introduce the following construction:

\begin{cns}\lb{cns02}
    Let $p$ be an odd prime  and $f(x),g(x),h(x)\in\mathbb F_p[x]$ be three polynomials
    of degree $\le k$.
    Define $E_{f,g,h} = \{e_1,\dots,e_p\}\in\{-1,+1\}^p$ by:
\begin{equation}
e_n=\left\{
\begin{array}{ll}
  \zj{\frac{f(n)}{p}}, & \textup{if } \zj{\frac{h(n)}{p}}\in\{0,1\}
                         \textup{ and } p\nmid f(n)g(n)\\
  \zj{\frac{g(n)}{p}}, & \textup{if } \zj{\frac{h(n)}{p}}=-1
                         \textup{ and } p\nmid f(n)g(n)\\
  1, & \textup{if } p\mid f(n)g(n).\\
\end{array}
\right.\lb{haha}
\end{equation}
\end{cns}  

Note that if $f=g$, then this construction coincides with
Construction \ref{cns01}.

In this construction, if $p\nmid f(n)g(n)h(n)$, then the following
formula can be proved:
\begin{equation}
  e_n=\frac{1}{2}
  \zj{1+\zj{\frac{h(n)}{p}}}\zj{\frac{f(n)}{p}}
  +\frac{1}{2}
  \zj{1-\zj{\frac{h(n)}{p}}}\zj{\frac{g(n)}{p}}.
  \lb{form01}
\end{equation}
From this formula, we will prove the following 
using multiplicative character techniques (e.g., Weil
theorem):

\begin{thm}\lb{thm01}
  Let $p$ be an odd prime and $f(x),g(x),h(x)\in\mathbb F_p[x]$ be
  three polynomials
  of degrees between $1$ and $k$, that have no multiple roots. Also assume that
  \begin{equation}
    f(x)\nmid \prod_{t=1}^{p} g(x+t)h(x+t)\ \ \textup{ and }\ \ 
    g(x)\nmid \prod_{t=1}^{p} h(x+t).\lb{kep01}
  \end{equation}
  Define the sequence $E_{f,g,h} = \{e_1,\dots,e_p\}\in\{-1,+1\}^p$ by
  Construction \ref{cns02}. Then,
\begin{equation}
  W(E_{f,g,h})\le 10kp^{1/2}\log p.
  \lb{h1}
\end{equation}
Assume that one of the following three conditions for $\ell$,
which is the order of the correlation, holds:\\ 
(i) $\ell=2$;\\
(ii) $\ell<p$ and 2 is a primitive root modulo p;\\
(iii) $(4k)^{\ell}<p$.\\
Then,
\begin{equation}
C_{\ell}(E_{f,g,h})\le 2^{\ell+3}\ell k p^{1/2}\log p. \lb{h2}
\end{equation}
\end{thm}
For reasons of symmetry, the theorem also holds even if condition
\eqref{kep01} is replaced by 
$g(x)\nmid \prod_{t=1}^{p} f(x+t)h(x+t)$ and
$f(x)\nmid \prod_{t=1}^{p} h(x+t)$.

\bigskip\noindent\textbf{Security considerations.}
The construction $E_{f,g,h}$ was specifically designed to counter potential attacks aimed at reconstructing the underlying polynomial of a Legendre symbol sequence. In the original construction proposed by Goubin, Mauduit, and Sárközy \cite{GMS}, a single polynomial $f(n)$ is used, and the sequence elements are defined as $e_n = \left(\frac{f(n)}{p}\right)$. If an attacker can determine the values of $e_n$ for a sufficient number of indices (roughly $\deg(f)$ values), they might be able to reconstruct the polynomial $f(x)$ using interpolation or other specialized algebraic algorithms.

In our proposed construction $E_{f,g,h}$, the attacker faces a significantly more complex task due to the following factors:

\begin{itemize}
    \item \textbf{Triple Uncertainty:} The attacker does not know which of the three polynomials ($f, g,$ or $h$) is responsible for a given element $e_n$. To even begin a reconstruction, one would first need to distinguish which indices $n$ satisfy $\left(\frac{h(n)}{p}\right) = 1$ and which satisfy $\left(\frac{h(n)}{p}\right) = -1$.
    \item \textbf{Hidden Switching:} The ``switching'' polynomial $h(n)$ is itself hidden. Since $e_n$ only reveals information about $f(n)$ or $g(n)$, the values of the Legendre symbol $\left(\frac{h(n)}{p}\right)$ are not directly observed. This adds an additional layer of protection, as the selector sequence is not public.
    \item \textbf{Combinatorial Explosion:} If an attacker attempts to guess the partition of $N$ observed values into those coming from $f(n)$ and those from $g(n)$, they encounter a combinatorial explosion. For $N$ observed elements, there are $2^N$ possible assignments. Without knowing the correct assignment, standard interpolation techniques for $f$ and $g$ cannot be applied effectively.
\end{itemize}

Thus, the interlacing of three polynomials ensures that even if a small number of values are leaked, the underlying algebraic structure remains computationally difficult to recover. This provide a significant security upgrade over the single-polynomial case while maintaining the same level of pseudorandomness.

\bigskip
Thus, even though Construction \ref{cns02}
is slightly more complicated than Construction \ref{cns01},
it still provides strong
bounds for the pseudorandom measures.

It is also clear that condition \eqref{kep01} cannot be completely
dropped from Theorem \ref{thm01}. For instance, if  
$f= h$ and $g= nh$, where $n\in\mathbb F_p$ is a quadratic non-residue,
then the elements of the sequence
$E_{f,g,h}$ are all $1$.
The strength of Construction \ref{cns02} will be further supported by numerical
calculations, and we will compare the pseudorandom measures of some sequences
in Constructions \ref{cns01} and \ref{cns02}.

At first glance, checking condition \eqref{kep01} may be inconvenient.
However, for small primes $p$, it does not require extensive computation and can be done
using polynomial division.
There are several ways to avoid this polynomial division,
one of which is to use only irreducible polynomials of the form
\[
x^{r}+a_{r-2}x^{r-2}+a_{r-3}x^{r-3}+\dots+a_1x+a_0  
\]
(so the coefficient of $x^{r-1}$ is $0$).
Another possibility is to choose $f$, $g$, and $h$ to all be products of
second-degree irreducible polynomials.
We will prove the following
\begin{thm}\lb{thm02}
  Let $p$ be prime, and $\mc A$, $\mc B$, $\mc C$
  be sets containing only quadratic non-residues modulo $p$
  for which
  \[
    \mc{A}\not\subseteq \mc{B}\cup \mc{C}\ \ \textit{ and }
    \ \  \mc{B}\not\subseteq \mc{C}.
  \]
  The polynomials $f$, $g$, and $h$ are defined as follows.
  \begin{align}
    f(x)=\prod_{n\in \mc{A}} (x^2-n),\
    g(x)=\prod_{n\in \mc{B}} (x^2-n),\
    h(x)=\prod_{n\in \mc{C}} (x^2-n).\lb{xy1}
  \end{align}  
 Define the sequence $E_{f,g,h} = \{e_1,\dots,e_p\}\in\{-1,+1\}^p$ by Construction \ref{cns02}.
 Then,
\begin{align}
  W(E_{f,g,h})&\le 10kp^{1/2}\log p\lb{xy2}\\
  C_{\ell}(E_{f,g,h})&\le 2^{\ell+3}\ell k p^{1/2}\log p.\lb{xy3}
\end{align}
\end{thm}

\bigskip\noindent\textbf{Remark.} The sequence 
$E_{f,g,h}$ given in Theorem \ref{thm02} is symmetric, as $f$, $g$
and $h$ are even polynomials $(f(x)=f(-x),\ g(x)=g(-x),\ h(x)=h(-x))$.
This implies that
for every element of the sequence $e_n=e_{p-n}$ holds.
Thus, for applications of the sequence given in Theorem \ref{thm02}, we suggest using at most the first $(p+1)/2$ elements.

\bigskip The analysis of Construction \ref{cns02} prompted us to consider
the generalizability of its underlying principle:  generating a sequence
by combining elements from multiple source sequences. This line of
questioning led to the development of the following construction:

\begin{cns}\lb{cns03}
  Let $\mc F\subset \{-1,+1\}^N$ be a large family of binary sequences. By taking binary sequences $F_N=\{f_1,f_2,\dots,f_N\}$, $G_N=\{g_1,g_2,\dots,g_N\}$
and $H_N=\{h_1,h_2,\dots,h_N\}\in\mc F$,
we can define a new sequence $E_N=\{e_1,e_2,\dots,e_N\}\in\{-1,+1\}^N$
with the following formula:
\begin{equation}
e_n=\left\{
\begin{array}{ll}
f_n, & \textup{if } h_n=1\\
g_n, & \textup{if } h_n=-1.
\end{array}
\right.\lb{tq1}
\end{equation}

\end{cns}
Similarly to Construction \ref{cns02}, there are special cases where the new sequence has weak pseudorandom properties. For example, if for every  $n$,
$f_n=h_n$ and $g_n=-h_n$,
then all elements of our sequence are $1$. To avoid such extreme cases, we must assume something about the large family $\mc F$ involved in the construction. For this, we will need the so-called cross-correlation measure. This
family measure was introduced by Mauduit, Sárközy, and
the first author of the present paper in \cite{GyMScrossi}.

\begin{dfn}
Let $N\in\mathbb N$, $\ell\in\mathbb N$, and for any $\ell$
binary sequences $E_N^{(1)},\dots,E_N^{(\ell)}$ with
\[
E_N^{(i)} =
\zj{e_1^{(i)},\dots,e_N^{(i)}}\in\{-1,+1\}^N
\ \textup{(for }i=1,2,\dots,\ell\textup{)}
\]
and any $M\in\mathbb N$ and $\ell$-tuple $D=(d_1,\dots,d_\ell)$
of non-negative integers with 
\begin{equation}
0\le d_1\le\dots\le d_\ell<M+d_\ell\le N,
\lb{a05}
\end{equation}
write 
\begin{equation}
V_\ell \zj{E_N^{(1)},\dots,E_N^{(\ell)}, M,D}
=\sum_{n=1}^{M} 
e_{n+d_1}^{(1)}\cdots e_{n+d_\ell}^{(\ell)}
\lb{a06}
\end{equation}
Let
\begin{equation}
\tilde{C}_\ell \zj{E_N^{(1)},\dots,E_N^{(\ell)}}
=\max_{M,D} \ab{V_\ell \zj{E_N^{(1)},\dots,E_N^{(\ell)}, M,D}}
\lb{a07}
\end{equation}
where the maximum is taken over all $D=(d_1,\dots,d_\ell)$
and $M\in\mathbb N$ satisfying \eqref{a05} with the 
additional restriction that 
if $E_N^{(i)}=E_N^{(j)}$ for some $i\ne j$, then we must not have 
$d_i=d_j$. Then the \underline{cross-correlation measure of order $\ell$}
of the family $\mc F$ of binary sequences $E_N\in\{-1,+1\}^N$ is defined 
as 
\begin{equation}
\Phi_\ell (\mc F)=\max \tilde{C}_\ell
\zj{E_N^{(1)},\dots,E_N^{(\ell)}}
\lb{a08}
\end{equation}
where the maximum is taken over all $\ell$-tuples of binary 
sequences $\zj{E_N^{(1)},\dots, E_N^{(\ell)}}$ with 
\begin{equation*}
E_N^{(i)}\in\mc F\textup{ for }i=1,\dots,\ell.
\end{equation*}
\end{dfn}

Then we will prove the following: 

\begin{thm}\lb{thm03}
Let  $\mc F\subset \{-1,+1\}^N$ be a large family of binary sequences.
For three distinct sequences $F_N,G_N,H_N\in \mc F$ define the sequence
$E_N\in\{-1,+1\}^N$ by Construction \ref{cns03}. 
For this new sequence, we have 
\begin{align*}
C_{\ell} (E_N) &\le  2^{\ell} \max_{\ell \le k\le 2\ell}\Phi_{k} (\mc F).
\end{align*}
\end{thm}
This theorem is particularly useful when the sequences $F_N,G_N,H_N$
are chosen from a large family $\mc F\subset\{-1,+1\}^N$ for which
$\max_{1 \le k\le 2\ell}\Phi_k(\mc F)\ll N^{1/2+\varepsilon}$. Such families
can be found, for example, in \cite{DoganSahin}, \cite{Gyarmati2018},
\cite{GyMScrossi}, \cite{Liu2024}, and \cite{WinterhofOguz2016}.

\section{Numerical Calculations}

While Theorems \ref{thm02} and \ref{thm03} provide rigorous theoretical upper bounds for the pseudorandom measures, these bounds are derived using the Weil theorem and often include large constants.
The primary objective of this section is to demonstrate that the actual pseudorandom measures of our construction are significantly smaller than the theoretical upper bounds derived from the Weil theorem. 
Regarding the family of sequences, we note that the construction $E_{f,g,h}$ allows for a large variety of sequences by choosing different triples of polynomials. While the detailed analysis of the cross-correlation measure of such a family is a challenging problem and lies beyond the scope of the present paper, the structure of the construction suggests that sequences generated by different polynomials will remain nearly orthogonal.

\bigskip
In the following, we determine and compare the measures
$W$ and $C_2$
for a few specific triples of polynomials $(f,g,h)$ and prime numbers $p$.
The sequences $E_f,\ E_g$ and $E_h$
are defined according to
Construction \ref{cns01} using specifically given polynomials $f$, $g$,
and $h$ in place of $f$ in \eqref{haha}. Meanwhile, $E_{f,g,h}$
is the sequence from Construction \ref{cns02}.
Thus, here are the four tables of results:

\bigskip
\begin{samepage}
\noindent\textbf{Example 1.}
Let the three polynomials be defined by
\[
  f(x) = x^2 + 1,\  g(x) = x^2 + 3x + 1,\ h(x) = x^3 - 1.
\]
Then, 
\begin{table}[H]
  \centering
\scalebox{0.84}[0.9]{
\begin{tabular}{|c|c|c|c|c|c|c|c|c|}
\hline
$p$ & $W(E_f)$ & $W(E_g)$ & $W(E_h)$ & $W(E_{f,g,h})$ & $C_2(E_f)$ & $C_2(E_g)$ & $C_2(E_h)$ & $C_2(E_{f,g,h})$ \\ 
\hline
2003 & 50 & 55 & 53 & 59 & 177 & 182 & 136 & 122 \\ 
3001 & 51 & 108 & 129 & 60 & 174 & 194 & 138 & 183 \\ 
4001 & 139 & 151 & 102 & 151 & 247 & 273 & 200 & 192 \\ 
5003 & 67 & 86 & 90 & 92 & 348 & 264 & 190 & 269 \\ 
6007 & 106 & 84 & 97 & 116 & 292 & 348 & 274 & 237 \\ 
\hline
\end{tabular}}
\end{table}
\end{samepage}

\bigskip
\begin{samepage}
\noindent\textbf{Example 2.}
Let the three polynomials defined by 
\[
  f(x) = x^2 + x + 1,\ g(x) = x^3 - x + 1\textup{ \ and \ }
  h(x) = x^4 + x - 1.
\]
Then, 
\begin{table}[H]
\centering
\scalebox{0.84}[0.9]{
\begin{tabular}{|c|c|c|c|c|c|c|c|c|}
\hline
$p$ & $W(E_f)$ & $W(E_g)$ & $W(E_h)$ & $W(E_{f,g,h})$ & $C_2(E_f)$ & $C_2(E_g)$ & $C_2(E_h)$ & $C_2(E_{f,g,h})$ \\ 
\hline
2003 & 48 & 66 & 72 & 56 & 169 & 136 & 139 & 117 \\ 
3001 & 75 & 133 & 51 & 184 & 203 & 147 & 152 & 172 \\ 
4001 & 187 & 45 & 192 & 110 & 266 & 189 & 206 & 159 \\ 
5003 & 72 & 101 & 81 & 79 & 274 & 209 & 220 & 195 \\ 
6007 & 147 & 124 & 131 & 104 & 294 & 211 & 204 & 226 \\ 
\hline
\end{tabular}}
\end{table}
\end{samepage}

\bigskip
\begin{samepage}
\noindent\textbf{Example 3.}
Let the three polynomials defined by 
\[
f(x) = x^4 - 1,\ g(x) = x^6 - 4x^3 + 3,\textup{ \ and \ }
h(x) = x^3 -6x^2 + 15x - 14.
\]
Then, 
\begin{table}[H]
\centering
\scalebox{0.84}[0.9]{ 
\begin{tabular}{|c|c|c|c|c|c|c|c|c|}
\hline
$p$ & $W(E_f)$ & $W(E_g)$ & $W(E_h)$ & $W(E_{f,g,h})$ & $C_2(E_f)$ & $C_2(E_g)$ & $C_2(E_h)$ & $C_2(E_{f,g,h})$ \\ 
\hline
2003 & 51 & 47 & 53 & 54 & 178 & 121 & 164 & 105 \\ 
3001 & 186 & 80 & 132 & 146 & 254 & 181 & 171 & 190 \\ 
4001 & 212 & 98 & 54 & 170 & 217 & 210 & 171 & 200 \\ 
5003 & 62 & 69 & 84 & 70 & 281 & 165 & 183 & 244 \\ 
6007 & 88 & 149 & 130 & 74 & 293 & 192 & 173 & 231 \\ 
\hline
\end{tabular}}
\end{table}
\end{samepage}

\bigskip
\begin{samepage}
\noindent\textbf{Example 4.}
Let the three polynomials defined by 
\[
f(x) = x^2 - 1,\ g(x) = x^3 + x^2 + 1 \textup{ \ and \ }
h(x) = x^4 + x^3 + 1.
\]
Then, 
\begin{table}[H]
\centering
\scalebox{0.84}[0.9]{
\begin{tabular}{|c|c|c|c|c|c|c|c|c|}
\hline
$p$ & $W(E_f)$ & $W(E_g)$ & $W(E_h)$ & $W(E_{f,g,h})$ & $C_2(E_f)$ & $C_2(E_g)$ & $C_2(E_h)$ & $C_2(E_{f,g,h})$ \\ 
\hline
2003 & 35 & 64 & 55 & 62 & 201 & 126 & 125 & 132 \\ 
3001 & 120 & 62 & 78 & 94 & 206 & 178 & 157 & 166 \\ 
4001 & 198 & 185 & 118 & 109 & 236 & 187 & 225 & 213 \\ 
5003 & 74 & 77 & 79 & 83 & 276 & 223 & 233 & 202 \\ 
6007 & 126 & 83 & 142 & 115 & 348 & 226 & 242 & 268 \\ 
\hline
\end{tabular}}
\end{table}
\end{samepage}

We found that in every case we studied, we have
\begin{align*}
  W(E_{f,g,h})&\le 2\max\{W(E_f),W(E_g),W(E_h)\}
  \intertext{and}
  C_2(E_{f,g,h})&\le 2\max\{C_2(E_f),C_2(E_g),C_2(E_h)\}.
\end{align*}

The computational results demonstrate that, apart from a few exceptional cases, the sequences  $E_{f,g,h}$
have pseudorandom properties almost as good as those of the sequences from the original Construction \ref{cns01}.  It is conjectured that these constructions are cryptographically secure, as recovering the full sequence from partial knowledge would necessitate computationally expensive algorithms.

\section{Proofs}

\bigskip\noindent\textbf{Proof of Theorem \ref{thm01}}.
Consider a triple of numbers $a,b,t$ for which
\begin{align}
  W(E_{f,g,h}) = \ab{\sum_{j=1}^{t} e_{a+jb}}. \lb{qq1}
\end{align}
Let $\mc L$ be the following set:
\begin{align*}
 \mc L=\{n:\ p\mid f(n)g(n)h(n)\}.
\end{align*}
Since the degrees of the polynomials $f$, $g$, and $h$ are $\le k$, 
we have $\ab{\mc L}\le 3k$.
Furthermore, by \eqref{form01}, \eqref{qq1} and the triangle inequality,
\begin{align}
  W(E_{f,g,h})
  &\resizebox{0.9\linewidth}{!}{\mbox{$\le\ab{\sum\limits_{j=1}^{t} \frac{1}{2}
    \left(1 + \zj{\frac{h(a+jb)}{p}}\right) \left(\frac{f(a+jb)}{p}\right) + \frac{1}{2} \left(1 - \zj{\frac{h(a+jb)}{p}}\right)
    \left(\frac{g(a+jb)}{p}\right)}$}}
  \notag\\
  &+ 2 \sum_{a+jb \in \mathcal{L}} 1\notag\\
  &\le \frac{1}{2} \left|\sum_{j=1}^{t}
    \zj{\frac{f(a+jb)}{p}}\right|
    + \frac{1}{2} \left|\sum_{j=1}^{t} \zj{\frac{g(a+jb)}{p}}\right|\notag\\
    &+ \frac{1}{2} \left|\sum_{j=1}^{t} \zj{\frac{f(a+jb)h(a+jb)}{p}}\right|
    +\frac{1}{2} \left|\sum_{j=1}^{t} \zj{\frac{g(a+jb)h(a+jb)}{p}}\right|
    + 6k\lb{qq2}
\end{align}

The theorem of Weil \cite{Weil} on character sums and polynomials can be extended to incomplete sums using the Vinogradov method. In this
extended theorem, Winterhof \cite{Winterhof} optimized the value of
the constant factor, proving the following:

\begin{lem}\lb{lem01}
Suppose that $p$ is a prime, $\chi$ is a non-principal character modulo
$p$ of order $d$, $f \in \mathbb{F}_{p}[x]$ has
$s$ distinct roots in $\overline{\mathbb{F}}_{P}$,
and it is not a constant
multiple of the $d$-th power of a polynomial over $\mathbb{F}_{p}$.
Let $y$ be a real number with  $0 < y \le p$. Then for any $x \in \mathbb{R}$.
\[
\left|\sum_{x<n\le x+y}\chi(f(n))\right|<sp^{1/2}(1+\log p). 
\]
\end{lem}

Since none of the polynomials $f,g$ and $h$ have multiple roots, none of them are of the form $c\dot q^2$ where $c\in\mathbb F_p$
and 
$q\in\mathbb F_p[x]$.
We also know that the polynomials $fh$ or $gh$
could only be of the form 
$cq^2$ (where $c\in\mathbb F_p$ and $q\in\mathbb F_p[x]$)
if $f$ is a constant multiple of $h$ (for $fh$) or if $g$ is a constant multiple of $h$ (for $gh$).
However, this contradicts \eqref{kep01}. Thus, by applying Lemma \ref{lem01}, we get that
\begin{align}
\left| \sum_{j=1}^{t} \left( \frac{f(a+jb)}{p} \right) \right|, \left| \sum_{j=1}^{t} \left( \frac{g(a+jb)}{p} \right) \right| &< kp^{1/2}(1+\log p) \notag\\
\left| \sum_{j=1}^{t} \left( \frac{f(a+jb)h(a+jb)}{p} \right) \right|,
\left| \sum_{j=1}^{t} \left( \frac{g(a+jb)h(a+jb)}{p} \right) \right| &< 2kp^{1/2}(1+\log p). \lb{qq3}
\end{align}
By \eqref{qq1}, \eqref{qq2} and \eqref{qq3}
\begin{align*}
  W(E_{f,g,h}) &< 3 kp^{1/2}(1+\log p) + 6k < 10kp^{1/2}\log p,
\end{align*}                 
from which \eqref{h1} follows.

\bigskip Let's move on to the proof of \eqref{h2}. This inequality trivially
holds if $\ell k\ge \dfrac{p^{1/2}}{13\log p}$. Thus, for the rest of the proof, we assume that
\begin{align}
  \ell k< \dfrac{p^{1/2}}{13\log p}. \lb{elk}
\end{align}
Consider numbers $M$ and
$0 \le d_1 < d_2 < \dots < d_{\ell} < M+d_{\ell} \le p$
for which
\begin{align}
C_{\ell}(E_{f,g,h})=
  \ab{\sum_{n=1}^{M}e_{n+d_{1}}...e_{n+d_{\ell}}}.
\lb{qq4}  
\end{align}  
Let $\mathcal{H}$ be the following set
\[
  \mathcal{H}=\{n : \exists d_{i} \text{ such that }
  p|f(n+d_{i})g(n+d_{i})h(n+d_{i})\}
\]
Then, 
$|\mathcal{H}|\le 3k\ell$.
By \eqref{form01} and \eqref{qq4} we get
\begin{align}
  C_{\ell}(E_{f,g,h})
  &= \Bigg| \frac{1}{2^{\ell}}\sum_{n=1}^{M} \prod_{j=1}^{\ell} \Bigg(
    \zj{\frac{f(n+d_{j})}{p}}
    +\zj{\frac{g(n+d_{j})}{p}} \notag\\
  &\qquad\qquad\qquad\qquad
    +\zj{\frac{f(n+d_j)h(n+d_j)}{p}}
    -\zj{\frac{g(n+d_j)h(n+d_j)}{p}}
    \Bigg) \Bigg| \notag \\
  &\quad + 2\sum_{n\in\mathcal{H}}1\notag\\
  &\le \Bigg| \frac{1}{2^{\ell}}\sum_{n=1}^{M} \prod_{j=1}^{\ell} \Bigg(
    \zj{\frac{f(n+d_{j})}{p}}
    +\zj{\frac{g(n+d_{j})}{p}} \notag\\
  &\qquad\qquad\qquad\qquad
    +\zj{\frac{f(n+d_j)h(n+d_j)}{p}}
    -\zj{\frac{g(n+d_j)h(n+d_j)}{p}}
    \Bigg) \Bigg| \notag \\
  &\quad + 6k\ell.
  \lb{qq5}
\end{align}
Expanding the product, we get a sum of $4^{\ell}$
Legendre symbol sums of the form:
\begin{align*}
  \sum_{n=1}^{M}\zj{\frac{f(n+\widetilde{d}_1)\dots f(n+\widetilde{d}_i)g(n+d_1')\dots g(n+d_v')
  h(n+d_1'')\dots h(n+d_z'')}{p}},
\end{align*}
where
\[
  i,v,z\le \ell,\ i+v+z \le 2\ell
\]  
and
$\widetilde{d}_r \ne \widetilde{d}_s, d_r' \ne d_s', d_r''\ne d_s''$.
Let $G$ be the following set of polynomials:
\begin{align*}
G = \{&F:\ \resizebox{0.87\linewidth}{!}{\mbox{$F=f(x+\widetilde{d}_1)\dots f(x+\widetilde{d}_i)g(x+d_1')\dots g(x+d_v')
  h(x+d_1'')\dots h(x+d_z'')$}},\\
  &\resizebox{0.91\linewidth}{!}{\mbox{$\text{where }
  i,v,z\le \ell,\ i+v+z \le 2\ell  \text{ and  } 
  \widetilde{d}_r \ne \widetilde{d}_s, d_r' \ne d_s', d_r''\ne d_s''$}}
  \}.
\end{align*}
Then by \eqref{qq5} 
\begin{align}
  C_\ell(E_{f,g,h}) \le \frac{4^\ell}{2^\ell} \max_{F \in G} \left| \sum_{n=1}^M
  \zj{\frac{F(n)}{p}} \right| + 6k\ell.\lb{bp}
\end{align}
In order to apply Lemma \ref{lem01}, we will use the following lemma.

\begin{lem}\lb{lemmac}
   Suppose that the conditions of Theorem \ref{thm01} hold. Then every
   polynomial $F\in G$ is not of the form $cq^2$ with
   $c\in\mathbb F_p$ and $q\in\mathbb F_p[x]$.
 \end{lem}  

\noindent\textbf{Proof of Lemma \ref{lemmac}}.
Consider a polynomial $F\in G$ of the form
\begin{align}
F=f(x+\widetilde{d}_1)\dots f(x+\widetilde{d}_i)g(x+d_1')\dots g(x+d_v')
  h(x+d_1'')\dots h(x+d_z''), \lb{FF}
\end{align}
 where $i,v,z\le\ell,\ i+v+z \le 2\ell$  and 
 $\widetilde{d}_r \ne \widetilde{d}_s, d_r' \ne d_s', d_r''\ne d_s''$. We will distinguish three cases.

\bigskip\noindent\textbf{Case I:} $i \ge 1$, i.e. in \eqref{FF}
the polynomial $F$ contains a factor $f(x+\widetilde{d}_r)$.

\bigskip\noindent In this case, since $f(x) \nmid \prod_{i=1}^n g(x+t) h(x+t)$,
$f(x)$ has an irreducible factor that does not divide
$\prod_{i=1}^n g(x+t) h(x+t)$. Let this irreducible factor be $f_0(x)$. Thus,
\begin{align}
  f_0(x) \nmid \prod_{t=1}^p g(x+t) h(x+t) \lb{qq6}
\end{align}  
Let us introduce the following equivalence relation:
Two irreducible polynomials $\phi(x), \psi(x) \in \mathbb{F}_p(x)$ are
equivalent if there exists a $\tau \in \mathbb{F}_p$ such that
$\phi(x) = \psi(x+\tau)$.
By \eqref{qq6}, the polynomials $g(x+d_i')$ and $h(x+d_i'')$ have no
irreducible factor that is equivalent to $f_0(x)$.
Let $\overline{f}(x)$ denote the product of the irreducible factors
of $f(x)$ that are equivalent to $f_0(x)$.

Then the product of the irreducible factors of $f(x+\widetilde{d}_i)$ that are equivalent to $f_0(x)$ is $\displaystyle{\overline{f}}(x+\widetilde{d}_i)$.
Thus, the product of the irreducible factors of $f(x+\widetilde{d}_1)\dots f(x+\widetilde{d}_i)$ that are equivalent to $f_0(x)$ is $\displaystyle{\overline{f}}(x+\widetilde{d}_1)\displaystyle{\overline{f}}(x+\widetilde{d}_2)\dots\displaystyle{\overline{f}}(x+\widetilde{d}_i)$.
Suppose that, contrary to Lemma \ref{lemmac},
the polynomial
$F(x)=f(x+\widetilde{d}_1)\dots f(x+\widetilde{d}_i)g(x+d_1')\dots g(x+d_v')h(x+d_1'')\dots h(x+d_z'')$ is of the form $cq^2$. This implies that
$\displaystyle{\overline{f}}(x+\widetilde{d}_1)\displaystyle{\overline{f}}(x+\widetilde{d}_2)\dots\displaystyle{\overline{f}}(x+\widetilde{d}_i)$ is also of
that form.
Then, define the sequence $\widetilde{E}_p = \{\widetilde{e}_1,\dots,\widetilde{e}_p\}$ by the formula
\begin{align}
\widetilde{e_n}\stackrel{\text{def}}{=}\left\{
\begin{array}{ll}
\left(\dfrac{\displaystyle{\overline{f}(n)}}{p}\right) &\text{ if } (n,p)=1 \\
1 & \text{ if } (n,p)>1
\end{array}
\right.\lb{tis}
\end{align}
then
\begin{align*}
  C_{i} (\widetilde{E}_p)
  &\ge\ab{\sum_{n=1}^{p}\widetilde{e}_{n+\widetilde{d}_1}\cdots\widetilde{e}_{n+\widetilde{d}_i}}\\
  &\ge\left|\sum_{n=1}^{p}\zj{\frac{\displaystyle{\overline{f}}(n+\widetilde{d}_{1})\cdot \displaystyle{\overline{f}}(n+\widetilde{d}_{i})}{p}}\right|-2ik\\
  &\ge p-3ik
    \ge p-3\ell k
\end{align*}    
However, $\displaystyle{\overline{f}}$ is a polynomial that satisfies the conditions of Theorem A,
and thus
\begin{align*}
  C_{i}(\widetilde{E}_p)&\le 10ikp^{1/2}\log p\\
  &\le 10\ell k p^{1/2}\log p.
\end{align*}  
Thus,
\begin{align*}
p-3\ell k&\le 10 \ell kp^{1/2}\log p\\
p&\le 10\ell kp^{1/2}\log p+3\ell k<13\ell kp^{1/2}\log p\\
\frac{p^{1/2}}{13\log p}&<\ell k,
\end{align*}
which contradicts \eqref{elk}. Thus, in Case I we proved that
the polynomial $F$ is not of the form $c q^2$,
where $c\in\mathbb F_p$ and $q\in\mathbb F_p[x]$.
 
\bigskip\noindent\textbf{Case II:} The polynomial $F$
does not contain a factor of the form
$f(x+d_r)$ and $v\ge 1$, so $F$
is of the form $g(x+d_1')\dots g(x+d_v')h(x+d_1'')\dots h(x+d_z'')$.

\bigskip\noindent In this case, $g(x)$ has an irreducible factor
that is not a divisor of $\prod_{t=1}^{p} h(x+t)$.
Let's call this irreducible factor $g_0(x)$.
Let $\overline{g}(x)$ denote the product of the irreducible factors
of $g(x)$ that are equivalent to $g_0(x)$.
Then the product of the irreducible factors of $g(x+d'_i)$ that are equivalent to $g_0(x)$ is $\displaystyle{\overline{g}}(x+d'_i)$.
Thus, the product of the irreducible factors of $g(x+d'_1)\dots g(x+d'_i)$ that are equivalent to $g_0(x)$ is $\overline{g}(x+d'_1)\overline{g}(x+d'_2)\dots\overline{g}(x+d'_i)$.

Similarly to Case I,
if the polynomial $F(x)=g(x+d'_1)\dots g(x+d'_i)h(x+d_1'')\cdots h(x+d_z'')$
is of the form $cq^2(x)$, then the product $\overline{g}(x+d'_1)\overline{g}(x+d'_2)\dots\overline{g}(x+d'_i)$ must also be of that form. We define $\widetilde{E}_p = \{\widetilde{e}_1, \dots, \widetilde{e}_p\}$ using the formula \eqref{tis} (but with $\overline{g}$ instead of $\overline{f}$).
This would imply, similarly to Case I, that $p-3\ell k \le C_i(\widetilde{E}_p) < 10 \ell kp^{1/2}(\log p)$, which is a contradiction. Thus, in this case as well, the polynomial $F$ cannot be of the form $cq^2(x)$, where $c \in \mathbb{F}_p$ and $q \in \mathbb{F}_p[X]$.

\bigskip\noindent\textbf{Case III:} The polynomial $F$
contains only factors of the form $h(x+d_r)$, so $F$
is of the form $h(x+d_1'')\dots h(x+d_z'')$.

\bigskip\noindent We define $\widetilde{E}_p
= \{\widetilde{e}_1, \dots, \widetilde{e}_p\}$ using the formula \eqref{tis} (but with $h$ instead of $\displaystyle{\overline{f}}$),
If $h(x+d_1'')\dots h(x+d_z'')$ is of the form $cq^2$ (with $c\in\mathbb F_p$,
$q\in\mathbb F_p[x]$) then $p-3\ell k\le C_z\zj{\widetilde{E}_p}$.
Using Theorem A we immediately get that
$C_{z} (\widetilde{E}_p)\le 10 z kp^{1/2}\log p$. Thus,
$p-3\ell k< 10\ell kp^{1/2}\log p$, which is a contradiction. 
Thus, we have proven in every case that $F(x)$
is indeed not a polynomial of the form
$cq^2$ with $c\in\mathbb F_p,\ q\in\mathbb F_p[x]$.
This completes the proof of Lemma \ref{lemmac}.

\bigskip By using Lemma \ref{lem01}, we get
\begin{align*}
  \left| \sum_{n=1}^M \zj{\frac{F(n)}{p}} \right| \le 2\ell k \sqrt{p} (1+\log p).
\end{align*}  
Thus, by \eqref{bp}
\begin{align*}
C_{\ell}(E_f,g,h)) &\le 2^{\ell + 1} \ell k p^{1/2} (1+\log p) + 6k\ell\\
 &< 2^{\ell+3} \ell k p^{1/2} \log p,
\end{align*}
which completes the proof of the theorem.

\bigskip\noindent\textbf{Proof of Theorem \ref{thm02}.}
We want to apply Theorem \ref{thm01} to the polynomials given
in \eqref{xy1}.
First, let's check if the conditions of \eqref{kep01} in Theorem \ref{thm01} are satisfied.
Since $\mc{A} \not\subseteq \mc{B}\cup\mc{C}$, there exists an
$n_1 \in \mc A$
such that $n_1 \notin \mc B \cup \mc C$. Then, we have
\begin{align}
  x^2 - n_1 \mid f(x) \quad \text{but} \quad x^2 - n_1 \nmid \prod_{t=1}^{T} g(x+t) h(x+t),\lb{xy4}
\end{align}
since the polynomial $x^2-n_1$ is irreducible and the irreducible factors
of the polynomial $g(x+t)h(x+t)$ are of the form $(x+t)^2-n$,
where $n \in \mc B \cup \mc C$.
Clearly, none of the irreducible polynomials $(x+t)^2-n=x^2-2tx+t^2-n$
are identical to the irreducible polynomial $x^2-n_1$.
Thus,
condition \eqref{kep01} indeed holds. Using Theorem \ref{thm01},
the result \eqref{xy2} follows immediately.

\bigskip To prove \eqref{xy3}, we will follow the notation and proof of
Theorem \ref{thm01}.
Let $G$ be the following set of polynomials:
\begin{align*}
G = \{&F:\ \resizebox{0.87\linewidth}{!}{\mbox{$F=f(x+\widetilde{d_1})\dots f(x+\widetilde{d_i})g(x+d_1')\dots g(x+d_v')
  h(x+d_1'')\dots h(x+d_z'')$}},\\
  &\resizebox{0.91\linewidth}{!}{\mbox{$\text{ where }
  i,v,z\le \ell,\ i+v+z \le 2\ell  \text{ and  } \widetilde{d_r} \ne \widetilde{d_s},
  d_r'\ne d_s', d_r''\ne d_s''$}}
  \}.
\end{align*}
Then by \eqref{qq5}, 
\begin{align}
  C_\ell(E_{f,g,h}) \le \frac{4^\ell}{2^\ell} \max_{F \in G} \left| \sum_{n=1}^M
  \zj{\frac{F(n)}{P}} \right| + 6k\ell.\lb{bp2}
\end{align}
In order to apply Lemma \ref{lem01}, we must show that for any $F\in G$,
the polynomial $F$ is not of the form $c q^2$,
where $c\in\mathbb F_p$ and $q\in\mathbb F_p[x]$. To do this,
let's consider a polynomial $F\in G$ of the form
\begin{align}
F=f(x+\widetilde{d_1})\dots f(x+\widetilde{d_i})g(x+d_1')\dots g(x+d_v')
  h(x+d_1'')\dots h(x+d_z''), \lb{FF2}
\end{align}
 where $i,v,z\le\ell,\ i+v+z \le 2\ell$  and 
 $\widetilde{d_r} \ne \widetilde{d_s}, d_r' \ne d_s', d_r''\ne d_s''$.
Next we distinguish three cases. 

\bigskip\noindent\textbf{Case I:} $i \ge 1$, i.e. in \eqref{FF}
the polynomial $F$ contains a factor $f(x+\widetilde{d_r})$.

\bigskip\noindent\textbf{Case II:} The polynomial $F$
does not contain factor $f(x+\widetilde{d_r})$ and $v\ge 1$, so $F$
is of the form $g(x+d_1')\dots g(x+d_v')h(x+d_1'')\dots h(x+d_z'')$.

\bigskip\noindent\textbf{Case III:} The polynomial $F$
contains only factors of the form $h(x+d_r)$, so $F$
is of the form $h(x+d_1'')\dots h(x+d_z'')$.

\bigskip
First, we prove that in Case I $F$ is not of the form $cq^2$,
where $c \in \mathbb{F}_p$ and $q \in \mathbb{F}_p[X]$. Let
$f_0(x)=x^2-n_1$, where $n_1\in\mc A$ but $n_1\notin\mc B\cup\mc C$.
By \eqref{xy4}, $g(x+d_1')\dots g(x+d_v') h(x+d_1'')\dots h(x+d_z'')$
has no irreducible factor equivalent to $f_0(x)$.
It's also clear that $f_0(x)$
is not equivalent to any other
irreducible factor of
$f(x)$. Then the irreducible factor of $f(x+\widetilde{d_r})$
that is equivalent to
$f_0(x)$ is $f_0(x+\widetilde{d_r})$.
Thus, the product of the irreducible factors of
$f(x+\widetilde{d_1})\dots f(x+\widetilde{d_i})$ that are equivalent
to $f_0(x)$ is
\begin{align}
  f_0(x+\widetilde{d_1})&f_0(x+\widetilde{d_2})\dots f_0(x+\widetilde{d_i})
                          \notag\\
                        &=\zj{(x+\widetilde{d_1})^2-n_1}\zj{(x+\widetilde{d_2})^2-n_1}\cdots
                          \zj{(x+\widetilde{d_i})^2-n_1}.
  \lb{tq}
\end{align}
Thus, if the polynomial $F(x)=f(x+\widetilde{d_1})\dots f(x+\widetilde{d_i})
g(x+d_1')\dots g(x+d_v')h(x+d_1'')\dots h(x+d_z'')$
is of the form $cq^2(x)$, 
then the polynomial in \eqref{tq} must also be of that form.
But this contradicts the fact that the irreducible polynomials
$(x+\widetilde{d_1})^2-n_1,\ (x+\widetilde{d_2})^2-n_1,\dots,
(x+\widetilde{d_i})^2-n_1$ are distinct.
This completes the proof for Case I, showing that
$F$ is not of the form $cq^2$, where $c \in \mathbb{F}_p$ and $q \in \mathbb{F}_p[X]$.

\bigskip Cases II and III can be handled similarly to Case I.
Thus, in all the three cases, we get that $F\in G$ is not of the form
$cq^2$ with $c\in\mathbb F_p$, $q\in\mathbb F_p[x]$.
Thus, we can apply
Lemma \ref{lem01}, from which
\[
\ab{\sum_{n=1}^{M}\zj{\dfrac{F(n)}{p}}}\le 2\ell kp^{1/2}(1+\log p)
\]
Thus, by \eqref{bp2}
\begin{align*}
C_{\ell}(E_{f,g,h}) &\le 2^{\ell + 1} \ell k p^{1/2} (1+\log p) + 6k\ell\\
 &< 2^{\ell+3} \ell k p^{1/2} \log p,
\end{align*}
which completes the proof of the theorem.

\bigskip\noindent\textbf{Proof of Theorem \ref{thm03}.}
Based on formula \eqref{tq1}, the elements of the sequence
$E_N=\{e_1,\dots,e_N\}$ satisfy
\begin{align}
  e_n&=\dfrac{1}{2}\zj{(1+h_n)f_n+(1-h_n)g_n}\notag\\
     &=\dfrac{1}{2}\zj{f_n+g_n+f_nh_n-g_nh_n}.\lb{t301}
\end{align}  

Consider numbers $M$ and
$0 \le d_1 < d_2 < \dots < d_{\ell} < M+d_{\ell} \le p$
for which
\begin{align}
C_{\ell}(E_{f,g,h})=
  \ab{\sum_{n=1}^{M}e_{n+d_{1}}\dots e_{n+d_{\ell}}}
\lb{t302}  
\end{align}  
By \eqref{t301} and \eqref{t302}
\begin{align}
  C_{\ell}(E_N)=
  \frac{1}{2^{\ell}}\ab{\sum_{n=1}^{M}
  \prod_{j=1}^{\ell}\zj{f_{n+d_j}+g_{n+d_j}+f_{n+d_j}h_{n+d_j}
  -g_{n+d_j}h_{n+d_j}}}.\lb{t303}
\end{align}
Expanding the product yields a sum of $4^{\ell}$
terms, each of the form:
\begin{align*}
  \sum_{n=1}^{M}f_{n+\widetilde{d}_1}\cdots f_{n+\widetilde{d}_i}
  g_{n+d_i'}\cdots g_{n+d'_v}h_{n+d''_1}\cdots h_{n+d''_{\ell}},
\end{align*}
where
\[
  i,v,z\le \ell,\ i+v+z \le 2\ell
\]  
and
$\widetilde{d}_r \ne \widetilde{d}_s, d_r' \ne d_s', d_r''\ne d_s''$. 
By the definition of the cross-correlation measure,
each of these sums has an absolute value
$\le \max_{\ell \le k\le 2\ell} \Phi_k(\mc F)$.
Thus, by \eqref{t303} we get
\begin{align*}
  C_{\ell} (E_N) \le \frac{4^{\ell}}{2^{\ell}}
  \max_{\ell \le k\le 2\ell} \Phi_k(\mc F)\\
  =2^{\ell}  \max_{\ell \le k\le 2\ell} \Phi_k(\mc F),
\end{align*}
which completes the proof.


\bigskip

\noindent Katalin Gyarmati\\
E\"otv\"os Lor\'and University, Institute of Mathematics, \\
H-1117 Budapest P\'azm\'any P\'eter s\'et\'any 1/C, Hungary    \\
Email: katalin.gyarmati@gmail.com

\bigskip

\noindent K\'aroly M\"ullner\\
E\"otv\"os Lor\'and University, Institute of Mathematics, \\
H-1117 Budapest P\'azm\'any P\'eter s\'et\'any 1/C, Hungary    \\
Email: mullni@student.elte.hu

\end{document}